\documentclass[11pt,a4paper]{article}
 \usepackage{mathrsfs}
\usepackage[utf8]{inputenc}
\usepackage{amsmath}
\usepackage{amssymb}
\usepackage{amsthm}
\usepackage{verbatim}
\numberwithin{equation}{section}
\usepackage{mathtools}
\usepackage{graphicx}
\usepackage[color=white]{todonotes}

\graphicspath{ {./images/} }

\newcommand{\R}{\mathbb R}

\newcommand{\N}{\mathbb N}
\newcommand{\Z}{\mathbb Z}

 \renewcommand{\>}{\right>}
\newcommand{\abs}[1]{\left\vert#1\right\vert} 
\newcommand{\ind}{1\mkern-7mu1}

\newcommand{\Var}{\operatorname{Var}}

\newcommand{\Rp}{\R_+}
\newcommand{\B}{\mathscr B}

\newcommand{\supp}{\operatorname{supp}}
\newcommand\norm[1]{\left\lVert#1\right\rVert}

 \usepackage{hyperref}
\usepackage[framemethod=tikz]{mdframed}

\usepackage{amsthm}

\theoremstyle{plain}
\newtheorem{thm}{Theorem}[section]
\newtheorem{lem}[thm]{Lemma}
\newtheorem{cor}[thm]{Corollary}
\newtheorem{prop}[thm]{Proposition}

\newtheorem*{theorem*}{Theorem}
\theoremstyle{definition}

\newtheorem{assumption}[thm]{Assumption}

\newtheorem*{asum*}{Assumption}

\newtheorem{rem}[thm]{Remark}

\renewcommand{\proof}{\vskip 0pt\noindent\textbf{\textit{Proof. }}}

\begin{document}
\title{\textbf{Large time behavior of critical marked Hawkes processes with heavy tailed marks and related branching particle systems}}
\author{\Large{Anna Talarczyk}\\
University of Warsaw, Institute of Mathematics\\ ul. Banacha 2, 02-097 Warsaw, Poland\\ \hbox{e-mail:}
annatal@mimuw.edu.pl
}
\bigskip
\date{May 4, 2026}
\maketitle
\begin{abstract}
We study large time behavior of critical
 marked Hawkes processes and related branching particle systems. 
In case of marked Hawkes processes we assume that the kernel function has multiplicative form and the marks corresponding to the events are nonnegative and are assigned independently from a common distribution. This distribution is in the normal domain of attraction of a $(1+\beta)$-stable law with $0<\beta<1$. Moreover, we assume that the mean number of events triggered by a single event is equal to $1$ (criticality). We show that, as the time is speeded up,  if $\beta$ is small enough then, the event counting process, appropriately normalized, converges to a spectrally positive $1/(1+\beta)$ stable L\'evy process. The convergence holds in law in the Skorokhod space of c\`adl\`ag functions equipped with $M_1$ topology.  We also study a borderline case.  The present paper complements the results of [A.Talarczyk:``A generalized central limit theorem for critical marked Hawkes processes'', arXiv:2504.11612], where the same model was studied in case of ``large'' $\beta$. We employ techniques involving a branching representation of marked Hawkes processes. This approach  allows to study  more general branching processes with branching mechanism in the normal domain of attraction of $(1+\beta)$-stable law.
\end{abstract}

\bigskip

\\

 \textup{2020} 

\section{Introduction}\label{sec:intro}

Hawkes processes and marked Hawkes processes are self-exciting processes that have a long history and many applications.
Hawkes processes were first introduced by Hawkes in 1971 in \cite{Hawkes1} and \cite{Hawkes2}, then they were studied and investigated by many authors, see e.g \cite{HawkesRev}, \cite{DaleyVereJones}, \cite{HorstXu},\cite{HorstXu_marked}, \cite{Laub} for an overview.

In our previous work \cite{AT} we investigated the central limit type for critical marked Hawkes processes as the time is speeded up.
In the present paper we study the large time behavior in the same model  but for a complementary range of parameters, that have not been discussed in \cite{AT}.

Large time behavior of Hawkes processes and their generalizations was studied in many papers.  To mention only a few most closely related papers: \cite{HawkesOakes}, \cite{BacryDelattreHoffmann} discuss CLT for Hawkes and multivariate Hawkes processes in the subcritical case, \cite{JaissonRosenbaum_nearly_unstable} and \cite{JR_2}  -- the nearly critical case, \cite{HorstXu} for Hawkes processes in the critical case.
We refer to \cite{AT} and the above mentioned papers for some background and literature.

In the present paper we consider marked Hawkes processes of the same type as in \cite{AT} but in the cases when the central limit type theorems do not hold.
We study  marked Hawkes processes with kernel functions of multiplicative form, that are probably the most widely used in applications.
We consider the critical case, in which a single event triggers on average one subsequent event.
Here we  briefly recall the model. We refer to \cite{AT} for more details.

Let $N=(N_t)_{t\ge 0}$ denote the a counting process, $N_0=0$ and let $\tau_j$, $j=1,2,\ldots$ be the times of jumps of the process $N$. They may be considered as times in which some events occur. We assume that each of these events is assigned a mark $\eta_j$, where $(\eta_j)_{j=1}^\infty$  are i.i.d. nonnegative random variables with common distribution  $\nu$ and, moreover, each $\eta_j$ is independent of the process $N$ up to time $\tau_j$. We assume that the intensity of the jumps of the process $N$ has the form

\begin{equation*}
 \Lambda(t)=\mu t+\sum_{0<\tau_i<t}\phi(t-\tau_i,\eta_i),
\end{equation*}
where $\mu>0$ denotes the baseline intensity (immigration intensity), 
and $\phi$ is a  kernel function. We additionally assume that it has multiplicative form
\begin{equation}
\phi(u,x)=x\varphi(u),
 \label{e:kernel}
\end{equation}
where
$\varphi$ is a fixed integrable function $\varphi:\R\mapsto \R_+=[0,\infty)$, $\varphi$ vanishes on $(-\infty,0)$. We consider the \textbf{critical} case
\begin{equation*}
 \norm{\varphi}_1:=\int_0^\infty \varphi(t)dt=1 
\end{equation*}
and 
\begin{equation}
 \int_0^\infty x\nu(dx)=1.
 \label{e:Enu}
\end{equation}
The case $\nu=\delta_1$ corresponds to the usual Hawkes process without marks.

In \cite{AT} we studied the limiting behavior of the process
\begin{equation}
 X_T(t)=\frac{N_{Tt}-EN_{Tt}}{F_T},\qquad t\ge 0
 \label{e:XTintro}
\end{equation}
as $T\to \infty$ and $F_T$ is an appropriate norming, chosen so that the limit process is nontrivial. 

We have shown that if $\varphi$ is such that
\begin{equation}
\label{e:phi_cond}
 \int_t^\infty \varphi(s)ds\sim t^{-\alpha} \qquad as \ t\to \infty
 \end{equation}
 for some $\alpha\in(0,1)$,
$\nu$ is such that
\begin{equation}
 \int_u^\infty \nu(dx)\sim u^{-(1+\beta)} \qquad as \ u\to \infty,
\label{e:nu_intro}
 \end{equation}
and additionally $\alpha<\beta$, then the appropriate norming is  
\begin{equation}
 F_T=T^{\frac{ 1+\alpha(2+\beta)}{1+\beta}}
 \label{e:FT_old}
\end{equation}
and the processes $X_T$ defined in \eqref{e:XTintro} converge in the sense of finite dimensional distributions to a $(1+\beta)$-stable process of the form
\begin{equation}
 C\int_{[0,t]} (t-s)^{\alpha}s^{\frac \alpha {1+\beta}}dL_{1+\beta}(s),
 \label{e:limit_proc_2}
\end{equation}
where $L_{1+\beta}$ is a $(1+\beta)$-stable L\'evy process totally skewed to the right (without negative jumps).
Under and additional, technical condition we also obtained convergence in law the space of c\`adl\`ag functions $D([0,\infty))$, equipped with the standard $J_1$ topology. We also proved an analogous theorem 
(with a Gaussian limit) in case when $\nu$ has finite variance. This case corresponds to formally setting $\beta=1$ in the formula for $F_T$ above and replacing $L_{1+\beta}$ by a Brownian motion. 
The latter theorem generalized an earlier result for not marked critical Hawkes processes obtained by Horst and Xu in \cite{HorstXu} (see Theorem 2.14 therein).

In \cite{AT} we were mostly concerned in CLT type theorems and the case when the long range dependence persists in the limit and in the investigation of this dependence. We also studied the corresponding empirical processes in the Schwartz space of tempered distributions  and discussed the role of the $\alpha$-stable potential in this setting.

The condition $\alpha<\beta$ was crucial, to obtain these limits. We suspected that the case $\alpha\ge \beta$ should be quite different. 
The present paper is devoted to the study of the latter case. We will see that, indeed, the behavior is different.

Before going to this case, recall that, in general, in the critical case, if condition \eqref{e:phi_cond} holds,
then we have the following behavior
of the mean:
\begin{equation*}
 \lim_{T\to \infty}T^{-1-\alpha}EN_{tT}=Ct^{\alpha+1},
\end{equation*}
where $C$ is a constant (see Proposition 2.6 in \cite{AT}). If additionally $\alpha<\beta$ and $F_T$ is given by \eqref{e:FT_old}, then
 $\lim_{T\to \infty}T^{-1-\alpha}F_T=0$ and the convergence of $X_T$ (Theorems 2.7 and 2.12 in \cite{AT}) implies that 
\begin{equation*}
\lim_{T\to\infty} T^{-1-\alpha}N_{tT} =\lim_{T\to\infty} T^{-1-\alpha}EN_{tT}=C t^{\alpha+1}
\end{equation*}
at least in probability.

The assumption that the tail of $\nu$ was relatively light with respect to the tail of $\varphi$ (i.e. that $\alpha<\beta$) was crucial here. The situation is different if $\alpha\ge \beta$. The limit is no longer deterministic. This is the object of the study of the present paper.

We consider the processes 
\begin{equation*}
 Y_T(t)=\frac 1{F_T}N(tT), \qquad t\ge 0
\end{equation*}
as $T\to \infty$ in cases $\alpha>\beta$ and $\alpha=\beta$.
In fact, in the case corresponding to $\alpha>\beta$ the assumption on $\varphi$ can be weakened. Instead of \eqref{e:phi_cond} with $\alpha>\beta$ we  assume that
there exists $\delta>0$ such that
\begin{equation*}
 \int_0^\infty t^{\beta+\delta}\varphi(t)dt<\infty.
\end{equation*}
We show that in this case the suitable norming is  $F_T=T^{1+\beta}$ and with this norming the processes $Y_T$ converge in the sense of finite dimensional distributions, as $T\to\infty$, to a $\frac 1{1+\beta}$-stable subordinator ($\frac 1{1+\beta}$-stable nondecreasing L\'evy process). Moreover,  we also prove functional convergence in the sense of convergence in law in the Skorokhod  space of c\`adl\`ag functions $D([0,\infty))$ equipped with $M_1$ topology (see Theorem
\ref{thm:alpha_ge_beta} below). Note, that in this case it is not possible to have convergence in $J_1$ topology, since the jumps of $Y_T$ are bounded by $1/F_T\to 0$, while the limit process has jumps.  

In the borderline case, assuming \eqref{e:phi_cond} and $\alpha=\beta$ the  norming is the same, $F_T=T^{1+\beta}$.
The convergence of $Y_T$ holds in the sense of finite distributions and in law in $D([0,\infty)$ equipped with $M_1$ topology. However, the limit process is now more involved.
 We show that in this case the limit process $Y_\infty$
 has the multidimensional Laplace transform of the following form:
 for any $m\in \N$ and $\theta_1,\ldots, \theta_m\ge 0$, $t_1,\ldots, t_m\ge 0$ 
\begin{equation*}
 E \exp\{-\sum_{j=1}^m\theta_j Y_{\infty}(t_j)\}=\exp\{- \mu
 \int_0^\infty v(t)dt\}, 
\end{equation*}
where $v$ is the unique nonnegative bounded solution with compact support of the equation
\begin{equation*}
 v(t)=C_1\int_0^\infty \left(\sum_{j=1}^\infty \theta_j\ind_{[0,t_j]}(t+s) -C_2 v^{1+\beta}(t+s)\right)s^{\beta-1}ds
\end{equation*}
with appropriate (known) constants $C_1$ and $C_2$ (see Theorem \ref{thm:critical}).

In Theorem \ref{thm:weakly_critical} we  also give an analogous result in case when $\nu$ has finite variance and $\int_0^\infty s\varphi(s)ds<\infty$. In this case the limit process is of the same form as the one obtained by Horst and Xu in \cite{HorstXu} (see Theorem 2.10 therein) for not marked Hawkes process, and may be described as integrated continuous state branching process with immigration.

As in \cite{AT}, we use a branching representation of the events in the marked Hawkes process. Because of a rather simple, multiplicative structure of the kernel function,  the points of jumps of the process $N$ may be obtained as the positions of particles in the following branching particle system on $\Rp$:
The immigrant particles arrive according to a Poisson process with parameter $\mu$, or, equivalently, their positions are the  points of a Poisson random measure on $\Rp$  with intensity $\mu\ell$, where $\ell$ denotes Lebesgue measure.
A particle present at site $\tau_I$ is
 independently assigned a mark $\eta_I$ from the distribution $\nu$. The particle has a random number $\rho_I$ of offspring, which, conditionally  on
$\eta_I=x$ is Poisson with parameter $x\int_{\Rp}\varphi (s)ds$. The children are placed at points $t_I+\xi_{I,k}$, and are assigned i.i.d. marks $\eta_{I,k}$, respectively, $k=1,\ldots, \theta_i$ where $(\eta_{I,k})_k$ are i.i.d. with law $\nu$ and
$(\xi_{I,k})_k$ are i.i.d with density $\frac 1{\norm{\varphi}_1}\varphi$, both sequences are independent of each other, and chosen independently for each of the particles. The offspring particles perform branching with displacements according to the same mechanism independently of all the other particles present in the system. The positions of all the particles are times of jumps of our process $N$.

To describe the branching mechanism with displacement we may equivalently say,
 that each particle has a random number of offspring, which is an independent copy of $\rho$ with $\rho$ having a generating function of the form
\begin{equation}
 G_{Hawkes}(s)=\int_0^\infty e^{-x(1-u)}\nu(du), \qquad u\in[0,1]
\label{e:generating_f}
 \end{equation}
and the children are placed to the right of the parent, the displacements being i.i.d. with density $\varphi$.

The condition \eqref{e:nu_intro} implies that the distribution of the number of offspring of a single particle (with generating function $G_{Hawkes}$) is in the normal domain of attraction of $(1+\beta)$-stable law.

It is then natural to consider extensions of this marked Hawkes model that do not require the special form \eqref{e:generating_f} of the generating function of the number of offspring, but allow to consider other generating functions that are in the domain of attraction of $(1+\beta)$-stable law. 
We formulate our results in this more general setting.

For example, one of the simplest branching mechanism that satisfies our assumptions and is  often considered in literature is
\begin{equation}
 G_\beta(s)=s+\frac {(1-s)^{1+\beta}}{1+\beta}.
 \label{e:Gbeta}
\end{equation}
This gives the same results as $G_{Hawkes}$.

We use the techniques specific to branching particle systems. In particular, we employ the formulas for Laplace transform of $Y_T$ already developed in \cite{AT}. However, it should be stressed that the analysis in this different regime is now also quite different.

The paper is organized as follows: In Section \ref{sec:result} describe the model and formulate our results. The proofs are given in Section \ref{sec:proofs}.

\section{Results}\label{sec:result}

\textbf{Notation and basic assumptions.} \label{basic_assumptions}

We will use the same notation as in \cite{AT} whenever possible.

Positive constants, whose exact value is irrelevant, are denoted by
 $C,C_1,C_2,\ldots$ and may change from line to line. Sometimes we change the index to indicate that the constant is different from the previous one.

We consider a branching particle system in $\Rp:=[0,\infty)$, described by the following parameters:
\begin{itemize}
 \item 
$\mu>0$ is the baseline, exogenous intensity. 
 \item
$\varphi:\R\mapsto \Rp$ is a Borel function vanishing on $(-\infty,0)$ and such that
\begin{equation}
\label{e:critical}
\norm{\varphi}_1:=\int_0^\infty \varphi (t)dt=1.
\end{equation}
\item $G$ is a generating function of a probability distribution on $\Z_+=\{0,1,\ldots\}$ with mean $1$. We exclude the trivial case of the distribution concentrated in $1$.
\end{itemize}
By $\rho$ we denote a generic random variable with generating function $G$, that is,
\begin{equation*}
 G(s)=E s^\rho,\qquad s\in [0,1],
\end{equation*}
with the assumption that $E\rho=1$.

By $\xi, \xi_1,\xi_2,\ldots$ we denote  i.i.d. nonnegative random variables with density $\varphi$.

\medskip 
The particle system is constructed in the following way:

Suppose that the exogenous ($0$-generation) particles are distributed according to a Poisson random measure on $\Rp$ with intensity $\mu\ell$, where $\ell$ denotes the Lebesgue measure. These may  be considered as immigrant particles.
Subsequently, each particle present in the system  independently gives rise to a random number of offspring with generating function $G$.
 The children particles, independently of each other and of all the other particles in the system,
 are displaced with respect to the parent 
 in the following way: given the parent $I$ at position $\tau_I$, the offspring particles are located at
$\tau_I+\xi_{I,j}$, $j=1,\ldots \rho_I$, where $\rho_I$ has generating function $G$,  $\xi_{I,j},$ $j=1,2, \ldots$ are i.i.d. with density $\varphi$ with the respect to the Lebesgue measure, they are also independent of $\rho_I$. The children particles then branch according to the same mechanism. All the branching and displacement events in the system are independent.

Note that since the branching mechanism is critical (has mean $1$) each of the families has a finite number of members almost surely.

We denote by $N$ the empirical measure of the system of all the immigrant particles and all their  offspring of all generations, that is, for $B\in \B(\Rp)$ (Borel subsets of $\Rp$)
\begin{equation}
 N(B)=\{\# \textrm{of particles of all generations in the set $B$}\}.
 \label{e:N}
\end{equation}
We also denote 
\begin{equation*}
 N(t):=N([0,t]), \qquad t\ge 0.
\end{equation*}

In case of the generating function of the form \eqref{e:generating_f} the process $N=(N_t)_{t\ge 0}$ corresponds to the event counting process in the marked Hawkes model with kernel function of the form \eqref{e:kernel} and with marks assigned independently from distribution $\nu$.

We are interested in the asymptotic behavior of the process
 \begin{equation}
  Y_T(t)=\frac 1{F_T}N(tT),\qquad t\ge 0,
\label{e:YT}
  \end{equation}
where $F_T$ is a suitable norming, chosen so that $Y_T$ converges in law to a nontrivial limit. 
  
\medskip
We will need some additional notation, we also have to impose some conditions on the system.

\medskip
We  denote 
\begin{equation}
  \label{e:Hdef}
  H(u)=G(1-u)-1+u \qquad u\in[0,1].
 \end{equation}

\medskip

In \cite{AT} the two important assumptions on $\varphi$ and $\nu$ in case of  $G$ of the form \eqref{e:generating_f}
were:
\begin{assumption} \label{ass:phi_alpha}
The Borel function $\varphi:\R\mapsto \Rp$ vanishes on $(-\infty,0)$  satisfies \eqref{e:critical} and there exists $\alpha\in(0,1)$
and $c_\varphi>0$ such that 
\begin{equation*}
 \lim_{t\to \infty } t^{\alpha}\int_t^\infty \varphi(s)ds=c_\varphi.
\end{equation*}
\end{assumption}

\begin{assumption} Suppose that $G$ is given by \eqref{e:generating_f} with a probability measure $\nu$  satisfying \eqref{e:Enu} and 
 \label{ass:nu_beta}. Moreover, assume that
 there exist $\beta\in(0,1)$ and $c_\nu>0$ such that
 \begin{equation}
  \lim_{x\to\infty}x^{1+\beta}\int_x^\infty \nu(dz)=c_\nu.
\label{e:nu_beta}
  \end{equation}
\end{assumption}
We also discussed the case of the finite variance and $\varphi$ satisfying Assumption \ref{ass:phi_alpha}.

 In \cite{AT} we have shown that Assumption \ref{ass:nu_beta} has the following consequence:

\begin{lem}
 \label{lem:H_properties_beta}
 Under Assumption \ref{ass:nu_beta}, the function   $H$ given by \eqref{e:Hdef} 
is nondecreasing and for any $x\in[0,1]$ we have
\begin{equation}
 \label{e:H_bound}
 H(x)= Kx^{1+\beta}W(x),
\end{equation}
where $K=\frac{c_\nu}{\beta}\Gamma(1-\beta)$ and $W$ is a nonnegative bounded function satisfying
\begin{equation}
 \lim_{x\to 0+}W(x)=1.
 \label{e:limW}
\end{equation}
Moreover, there exists a constant $C>0$ such that for any $0\le x\le y$
\begin{equation}
 \label{e:H_bound_2}
 0\le H(y)-H(x)\le C(y-x)y^\beta.
\end{equation}
\end{lem}

It turns out that these are the only important properties of $G$ that determine the large time behavior of the process $N$, therefore, in our more general setting we will use them as an assumption on the generating function $G$, corresponding to the number of offspring of a single particle.

\begin{assumption}
 \label{ass:beta_general}
 Assume that the distribution of the number of offspring of a single particle has mean $1$ and a generating function $G$ and that $H$ given by 
 \eqref{e:Hdef} satisfies  \eqref{e:H_bound}-\eqref{e:H_bound_2} with some positive constant $K$ in \eqref{e:H_bound}.
\end{assumption}

Note that $H$ is always nondecreasing if the branching mechanism has mean $1$.

As observed above, Assumption \ref{ass:nu_beta} implies that Assumption \ref{ass:beta_general} is satisfied. Also the function \eqref{e:Gbeta} satisfies Assumption \ref{ass:beta_general}. In the latter case $H$ is particularly simple, $H(s)=s^{1+\beta}/(1+\beta)$.

In the present paper we are interested in the case $\beta\le \alpha$. If the inequality is strict, i.e.  $\beta<\alpha$, it turns out that instead of Assumption \ref{ass:phi_alpha} it is enough to assume a weaker condition:
\begin{assumption}
 \label{ass:beta_integrable}
Let $\beta\in (0,1)$ and suppose that $\varphi$ is such that there exists $\delta>0$ such that
 \begin{equation}
 \int_0^\infty s^{\beta+\delta}\varphi(s)ds<\infty.
  \label{e:varphi_beta_integrable}
 \end{equation}
\end{assumption}

Clearly, if Assumption \ref{ass:phi_alpha} is satisfied with $\alpha>\beta$, then \eqref{e:varphi_beta_integrable} also holds.

\medskip

\textbf{The resolvent.}

We denote 
\begin{equation}
 R(s)=\sum_{k=1}^\infty \varphi^{(*k)}(s),
\label{e:R}
 \end{equation}
where $\varphi^{(*k)}$ denotes the $k$-th convolution. $R$ is called the resolvent and 
 plays an important role in the context of Hawkes processes and renewal processes. This is also the case in the present paper.  

Similarly as in \cite{HorstXu} we denote
\begin{equation}
 \label{e:IR_def}
 I_R(t)=\int_0^t R(s)ds.
\end{equation}

From (8.6.3) in \cite{Bingham} (see also \cite{HorstXu}, Proposition 2.5 (3)) 
it follows that under  Assumption \ref{ass:phi_alpha} we have
\begin{equation}
 \label{e:IR}
\lim_{T\to \infty}\frac 1{T^\alpha} I_R(T)=c_{\alpha,\varphi},
\end{equation}
where
\begin{equation}
 \label{e:c_alpha}
 c_{\alpha,\varphi}= \frac 1{c_\varphi\Gamma(1+\alpha)\Gamma(1-\alpha)}.
\end{equation}

\bigskip
We now proceed to our main results.

 \begin{thm}\label{thm:alpha_ge_beta}
 Suppose that $\beta\in(0,1)$ and the Assumptions \ref{ass:beta_general}  and \ref{ass:beta_integrable} are satisfied.
 Then the process $Y_T$ defined by \eqref{e:YT} with 
 \begin{equation}
  F_T=T^{1+\beta}\label{e:FT}
 \end{equation}
converges in the sense of finite dimensional distributions and in law in $D([0,\infty))$ equipped with $M_1$ topology to $K^{-1}L_{1/(1+\beta)}$,
where
$L_{ 1/(1+\beta)}$ is a $\frac 1{1+\beta}$-stable subordinator, that is $L_{1/1+\beta}$  is a nondecreasing process with stationary and independent increments such that its one dimensional distributions have the following Laplace transform:
\begin{equation*}
 E\exp\left\{-\theta L_{1/(1+\beta)}(t)\right\}=\exp\left\{-t\theta^{1/(1+\beta)}\right\}, \qquad t\ge 0, \theta\ge 0.
\end{equation*}
 \end{thm}
\begin{rem}a)
By  Proposition 2.6 of \cite{AT} it follows that  under assumptions Assumption \ref{ass:phi_alpha} 
we have
that 
\begin{equation*}
 \lim_{T\to \infty}\frac 1{T^{1+\alpha}}E{N(Tt)}= \frac{c_{\alpha,\varphi}}{1+\alpha}t^{1+\alpha}.
\end{equation*}
(This proposition was formulated for the marked Hawkes process, but from its proof it is clear that it holds in the more general setting as well, without any additional assumptions on $G$, other than that the corresponding law has mean $1$.)

Therefore, from Theorem 2.7 of \cite{AT}
it follows that under Assumptions \ref{ass:phi_alpha} and \ref{ass:beta_general} with $\alpha<\beta$ the suitable norming for $N(tT)$ is $F_T=T^{1+\alpha}$ and  
\begin{equation*}
 \lim_{T\to \infty}Y_T(t)= \frac{c_{\alpha,\varphi}}{1+\alpha}t^{1+\alpha}\qquad \textrm{in probability}.
\end{equation*}
In case $\alpha>\beta$ considered in Theorem \ref{thm:alpha_ge_beta}
the norming in $Y_T$ given by \eqref{e:FT} is smaller than $T^{1+\alpha}$ and  $EY_T(t)$ goes to infinity, even though $Y_T(t)$ converges in law. This is consistent with the fact that $EL_{1/(1+\beta)}(t)=+\infty$ for $t>0$.

b) Observe that in Theorem \ref{thm:alpha_ge_beta} there is no convergence in $J_1$ topology, since the jumps of $Y_T$ are bounded by $\frac 1{T^{1+\beta}}$ and the limit process is not continuous.
\end{rem}

To investigate the borderline case we need a finer Assumption \ref{ass:phi_alpha} we have the following theorem
\begin{thm}
 \label{thm:critical}
 Suppose that Assumptions \ref{ass:phi_alpha} and \ref{ass:beta_general} are satisfied with $\alpha=\beta<1$ and let $F_T=T^{1+\beta}$ as in \eqref{e:FT}. Then  the process $Y_T$ given by
 \eqref{e:YT} converges in the sense of finite dimensional distributions and in law in $D([0,\infty))$ equipped with $M_1$ topology to a nonnegative process $Y_\infty=(Y_\infty(t))_{t\ge 0}$ whose finite dimensional distributions are such that for any $m\in \N$, $\theta_1, \ldots, \theta_k\ge 0$ and $t_1,\ldots, t_m\ge 0$ we have 
 \begin{equation}
  E \exp\{-\sum_{k=1}^m\theta_kY_{\infty}(t_k)\}=\exp\{-\mu\int_0^\infty v_f(t)dt\},
  \label{e:Lapl_lim_critical}
 \end{equation}
where $v_f$ is a unique nonnegative bounded function with compact support satisfying the equation
\begin{equation}
 v_f(t)=\beta c_{\beta,\varphi}\int_0^\infty \left(f(t+s)-Kv_f^{1+\beta}(t+s)\right)s^{\beta-1} ds,
 \label{e:g_infty}
 \end{equation}
where
 \begin{equation}
  f(t)=\sum_{k=1}^m\theta_k\ind_{[0,t_k]}(t)
  \label{e:f_funct_thm}
 \end{equation}
and $K$ and $c_{\beta,\varphi}$ are given by \eqref{e:H_bound} and \eqref{e:c_alpha}, respectively.
\end{thm}

To complete the picture we also state the corresponding result in case when the  distribution of the number of offspring has finite variance and  $\varphi$ vanishes more quickly at $+\infty$. 
\begin{thm}
\label{thm:weakly_critical}
 Assume that $\varphi:\R\mapsto \R_+$ vanishes on $(-\infty,0)$, satisfies \eqref{e:critical} and 
 \begin{equation}
  \label{e:phi_x_integrable}
m_\varphi:=\int_0^\infty x\varphi(x)dx<\infty.
  \end{equation}
Moreover, assume that the number of offspring of a single particle has finite variance denoted by $\sigma_G^2$ and let 
\begin{equation}
 F_T=T^2\label{e:F_T_2}
\end{equation}
Then
 the process $Y_T$ given by
 \eqref{e:YT} converges in the sense of finite dimensional distributions and in law in $D([0,\infty))$ equipped with $M_1$ topology to a nonnegative process $Y^*_\infty=(Y^*_\infty(t))_{t\ge 0}$ whose finite dimensional distributions are
 such that for any $m\in \N$, $\theta_1, \ldots, \theta_k\ge 0$ and $t_1,\ldots, t_m\ge 0$ we have 
 \begin{equation}
  E \exp\{-\sum_{k=1}^m\theta_kY^*_{\infty}(t_k)\}=\exp\{-\mu\int_0^\infty v_f^*(t)dt\},
  \label{e:Lapl_lim_critical_2}
 \end{equation}
where $v_f^*$ is a unique nonnegative bounded function with compact support satisfying the equation
\begin{equation}
v_f^*(t)=\frac 1{m_\varphi}\int_0^\infty \left(f(t+s)-\frac{\sigma_G^2}2 (v_f^*(t+s))^2\right) ds,
 \label{e:g_infty_2}
 \end{equation}
where $f$ is given by \eqref{e:f_funct_thm}.
\end{thm}
\begin{rem} \label{rem:weakly_critical}
It is clear that the form of the limit process in Theorem \ref{thm:weakly_critical} is formally the same as in Theorem \ref{thm:critical} if we set $\beta=1$.

 Theorem \ref{thm:weakly_critical} corresponds to the weakly critical case studied in Theorem 2.10 of \cite{HorstXu} for the usual Hawkes processes (not marked). In fact,  Theorem 2.10 of \cite{HorstXu} gives more information on the joint convergence of intensity process and the process $Y_T$  and other related processes in the simpler Hawkes processes case. Theorem 2.10 of \cite{HorstXu} shows that the limit process $Y_\infty^*$ has the form
 \begin{equation}
  Y_\infty^*(t)=\int_0^t\Lambda^*(s)ds \label{e:r1}
 \end{equation}
Where $\Lambda^*(s)$ is a continuous state branching process with immigration (see e.g. \cite{Zenghu_Li}) satisfying the following stochastic equation
\begin{equation}
\Lambda^*(t)=\frac{\mu}{m_\varphi}t+\int_0^t\frac {\sigma_G}{m_\varphi} \sqrt{\Lambda^*(s)}dB(s), 
\label{e:r2}
\end{equation}
where $\sigma_G^2=1$ (for Hawkes process) and
 $B$ is a Brownian motion.

Using Theorem 2.11 of \cite{HorstXu} it is not difficult that representation \eqref{e:r1} and \eqref{e:r2} holds also in case of Theorem \ref{thm:weakly_critical} (see end of Section \ref{sec:proofs}).

Note, that the limit process $ Y^*_\infty$ is continuous hence $M_1$ convergence implies uniform convergence and $J_1$ convergence.  (cf. e.g. p.14 (Bonus property) in \cite{Kern} ).

Unfortunately we were not able to find a corresponding representation of  the limit process $Y_\infty$ in Theorem \ref{thm:critical}. It seems that the limit process should be related in some way to continuous state branching processes with branching function $x\mapsto C x^{1+\beta}$, but it appears that this dependence is more involved.
\end{rem}

\section{Proofs}
\label{sec:proofs}
\subsection{Properties of the resolvent}
\label{sec:Resolvent}

We start with the discussion of the resolvent $R$ and its integral $I_R$ defined in \eqref{e:R} and \eqref{e:IR_def}, respectively. We already know that under Assumption \ref{ass:phi_alpha} we have \eqref{e:IR}. In the next lemma we discuss the behavior of $I_R$ in other cases that we will be interested in.

\begin{lem}
 a) For any $\varphi$ satisfying \eqref{e:critical}
 \begin{equation}
  \lim_{T\to \infty }\frac 1T I_R(T)=\frac 1{m_\varphi},
  \label{e:IR_a_1}
 \end{equation}
where $m_\varphi$ is given by \eqref{e:phi_x_integrable}, with the convention $\frac 1{\infty}=0$. In particular we always have
 \begin{equation}
\limsup_{T\to \infty} \frac 1T I_R(T)<\infty\label{e:IR_a}.
\end{equation}
b)  Under Assumption \ref{ass:beta_integrable} we have
\begin{equation}
\liminf_{T\to \infty} \frac 1{T^{(\beta+\delta)\wedge 1}} I_R(T)>0.\label{e:IR_b}
  \end{equation}
\end{lem}
\proof 
It is a simple and standard result from renewal theory that if $\zeta_1,\zeta_2, \ldots$ are i.i.d. nonnegative random variables with $E\zeta_1\in(0,\infty]$, then
\begin{equation}
 \lim_{T\to \infty}\frac 1T\sum_{k=1}^\infty P(\zeta_1+\ldots+\zeta_k\le T)=\frac 1{E\zeta_1} 
 \label{e:renewal_mean}
\end{equation}
with the convention $\frac 1{\infty}=0$.
(see e.g. \cite{Durrett}, Theorem 4.4.2).

Recall that $\xi_1,\xi_2, \ldots$ denote i.i.d. random variables with density $\varphi$. Then by the definition of $I_R$ (see \eqref{e:IR_def} and \eqref{e:R}) and  \eqref{e:renewal_mean} applied for $\zeta_j=\xi_j$ we immediately obtain \eqref{e:IR_a}.

To prove \eqref{e:IR_b} observe that in Assumption \ref{ass:beta_integrable} without loss of generality we may assume that $\delta$ is such that $\beta+\delta\le 1$. Then, using $(x+y)^{\beta+\delta}\le x^{\beta+\delta}+y^{\beta+\delta}$ for $x,y\ge 0$ we have 
\begin{align*}
 \frac 1{T^{\beta+\delta}}I_R(T)=&\frac 1{T^{\beta+\delta}} \sum_{k=1}^\infty P((\xi_1+\ldots+\xi_k)^{\beta+\delta}\le T^{\beta+\delta})\\
 \ge &\frac 1{T^{\beta+\delta}} \sum_{k=1}^\infty P(\xi_1^{\beta+\delta}+\ldots+\xi_k^{\beta+\delta}\le T^{\beta+\delta}).
 \end{align*}
 Therefore by \eqref{e:renewal_mean} with $\zeta_j=\xi_{j}^{\beta+\delta}$ we obtain
 \begin{equation*}
\liminf_{T\to\infty} \frac 1{T^{\beta+\delta}}I_R(T)\ge\frac 1{E{\xi_1^{\beta+\delta}}}>0
 \end{equation*}
and \eqref{e:IR_b} follows.
 \qed

\subsection{Laplace transform of finite dimensional distributions}

We will use the formulas developed in Proposition 3.1 of \cite{AT}. Note that it was proved under only very basic assumptions on the particle system, it is also clear that the form of $G$ was not important. One only needs to assume that $G$ is a generating function of a probability distribution on $\Z_+$ with mean $1$.

Let us recall the notation used in \cite{AT}:

$N$ is the empirical measure given by \eqref{e:N}. We denote by $N^t$ the empirical measure of an analogous particle system without immigration and starting from a single ``root'' particle located at $t\ge 0$.

For a Borel function $f:\R\mapsto \R$ we write
\begin{equation*}
 \<N,f\>=\int_0^\infty f(t) N(dt)
\end{equation*}
whenever the integral makes sense.

For $T\ge 1$ and a Borel function $f:\R\mapsto \R$ we denote
\begin{equation}
\label{e:fT}
 f_T(t)=\frac 1{F_T}f(\frac t T), \qquad t\ge 0,
\end{equation}
where $F_T$ is given by \eqref{e:FT}.

We will work with $f$ of the form 
\begin{equation}
f(t)=    \sum_{k=1}^m \theta_k\ind_{[0, t_k]}.
\label{e:f_sum}
\end{equation}
We denote 
\begin{equation}a=\max_{k\le m}t_k.
 \label{e:a}
\end{equation}

Then
 \begin{equation*}
  \<N,f_T\>=\frac 1{F_T}\sum_{k=1}^m \theta_k N(Tt_k)=\sum_{k=1}^m \theta_k Y_T(t_k).
 \end{equation*}
With this notation it is clear that to prove convergence of finite dimensional distributions of $Y_T$ we need to show convergence in law of $\<N,f_T\>$ for any $f$ of the form \eqref{e:fT} with $\theta_1,\ldots, \theta_m\in \R$, $t_1,\ldots, t_k\in \Rp$. Since $Y_T$ is nonnegative we may work with Laplace transforms and by the standard argument, it suffices to show convergence of the Laplace transform of $\<N,f_T\>$ for any $f$ of the form \eqref{e:fT} with all $\theta_k$ nonnegative
to the corresponding Laplace transform of the limit process. 
Additionally in Theorem \ref{thm:critical} we have to prove existence of the limit process with this Laplace transform, i.e. for example tightness of the finite dimensional distributions of $Y_T$.

Proposition 3.1 of \cite{AT} implies the following: 

 \begin{prop}
 \label{prop:Laplace}
a) For any $m\in\N$ and $\theta_1,\ldots, \theta_m, t_1,\ldots, t_m\in (0,\infty)$ we have
\begin{equation}
 E\exp\left\{-\sum_{k=1}^m\theta_k Y_{T}(t_k)\right\}=\exp\left\{-\mu\int_0^\infty g_T(t)dt\right\},
 \label{e:Laplace_2}
\end{equation}
where
\begin{equation}
  g_{T}(t):=1-Ee^{-\<N^t, f_T\>},
\label{e:def_g}
 \end{equation}
 with $f_T$ given by \eqref{e:f_sum} and \eqref{e:fT} and recalling that $N^t$ denotes the empirical measure of a branching particle system starting from a single particle at point $t$ (with no immigration). Moreover, the function $g_T$ satisfies
\begin{align}
  g_{ T}(t)=&1-e^{-f_T(t)}
G\left(1-\int_0^\infty g_{T}(t+s)\varphi(s)ds\right)
 \label{e:g_eq}\\
 =& 1-e^{-f_T(t)}+e^{-f_T(t)}\int_0^\infty g_T(t+s)\varphi (s)ds -e^{-f_T(t)}H \left(\int_0^\infty g_T(t+s)\varphi(s) ds\right)
 \label{e:g_eqH}.
\end{align}
b) Let the function $h_T$ be given by
\begin{equation}h_T(t):=E\<N^t,f_T\>
 \label{e:h_T_def}
\end{equation}
Then $h_T$ 
satisfies the equation
\begin{equation}
\label{e:renewal}
 h_{T} (t) =  f_T(t)+\int_0^\infty  h_{T}(t+s)\varphi (s)ds.
\end{equation}
Moreover, $h_T$ may be expressed explicitly as
\begin{equation}
 h_T(t)=f_T(t)+\int_0^\infty f_T(t+s) R(s)ds,
 \label{e:hR}
\end{equation}
where $R$ is the resolvent defined in \eqref{e:R}.\\
c) The functions $g_T$ and $h_T$ satisfy
\begin{equation}
 0\le g_T(t)\le h_T(t)\qquad t\ge 0.
 \label{e:ineq_g_f}
\end{equation}
\end{prop}

Since  $f$ is given by \eqref{e:f_sum} with nonnegative $\theta_k$, it is nonincreasing and has compact support $[0,a]$, therefore  
an easy consequence of Proposition \ref{prop:Laplace} is the following corollary:
\begin{cor}
 Under the assumptions and notation of Proposition \ref{prop:Laplace} the functions $g_T$ and $h_T$ are nonincreasing, and have their supports in $[0,Ta]$, where $a$ is given by \eqref{e:a}, $g_T$ is bounded by $1$. Moreover, $g_T$ satisfies the equation
 \begin{equation}
  g_T(t)=V_T(t)+\int_0^\infty g_T(t+s)\varphi(s)ds
\label{e:gT_VT}
  \end{equation}
where $V_T$ is bounded and nonnegative, has its support in $[0,Ta]$ and it is given by
\begin{equation}
V_T(t)=1-e^{-f_T(t)}-(1-e^{-f_T(t)})\int_0^\infty g_T(t+s)\varphi(s)ds-e^{-f_T(t)}H\left(\int_0^\infty g_T(t+s)\varphi(s)ds\right).
 \label{e:def_VT}
\end{equation}
Moreover,  equation \eqref{e:gT_VT} may be written in an equivalent form as
\begin{equation}
 g_T(t)=V_T(t)+\int_0^\infty V_T(t+s)R(s)ds,
 \label{e:gT_VT_2}
\end{equation}
where $R$ is given by \eqref{e:R}.
\end{cor}
Note that $V_T=f_T-U_{f_T}$ where $U_{f_T}$ is given by formula (3.19) in \cite{AT}.
\proof 
It follows from the definition of $g_T$ (see \eqref{e:def_g}) that it is bounded by $1$. By \eqref{e:a} and \eqref{e:f_sum},
 the support of $f_T$ is in $[0,Ta]$ and  hence, if $t>T a$ then $N^t$ does not charge $[0,Ta]$, so $g_T(t)=0$ for $t>Ta$. Moreover,  
observe that for any $s\ge 0$ the random variable
$\<N^t,f_T\>$ has the same law as $\<N^{t+s}, f_T(\cdot -s)\ind_{[s,\infty)}\>$ and since
 $f_T$ is nonnegative and nonincreasing and $N^{t+s}$ does not charge $[0,s]$ we have
\begin{equation*}
 \<N^{t+s},f_T(\cdot-s)\ind_{[s,\infty)}\>\ge \<N^{t+s},f_T\ind_{[s,\infty)}\>=\<N^{t+s},f_T\>.
\end{equation*}
From  \eqref{e:def_g} and \eqref{e:h_T_def} it follows that both $g_T$ and $h_T$ are nonincreasing. This also implies in particular that 
\begin{equation}
 g_T(t)=\int_0^\infty g_T(t)\varphi(s)ds\ge \int_0^\infty g_T(t+s)\varphi(s)ds.
 \label{e:gT_ineq}
\end{equation}

\eqref{e:gT_VT} is a direct consequence of \eqref{e:g_eqH}. Moreover, from \eqref{e:gT_VT} and \eqref{e:gT_ineq} it follows that $V_T$ is nonnegative. The support of $V_T$ is contained in $[0,Ta]$. Boundedness of $V_T$ is also clear.

\eqref{e:gT_VT_2} follows by iteration of  \eqref{e:gT_VT} since $V_T$ is bounded and has compact support  and  since if $\xi_1,\xi_2, \ldots $ are i.i.d. with density $\varphi$, then $\xi_1+\ldots+\xi_n\to \infty$ a.s. if $n\to \infty$. 
\qed

\subsection{Properties of the function $g_T$.}

In the proof of convergence of finite dimensional distributions in Theorems \ref{thm:alpha_ge_beta}, \ref{thm:critical} and \ref{thm:weakly_critical}  will use Proposition \ref{prop:Laplace} and will have to investigate convergence of the integral of $g_T$. Some parts of this proof are the same in both cases. We will state them as lemmas in this section. In this section we  do not assume anything about $\varphi$, other than the basic assumption \eqref{e:critical}. 

\begin{lem}
 \label{lem:lem1}
 a) Suppose that Assumption \ref{ass:beta_general} is satisfied and that $f$ is of the form \eqref{e:f_sum} with all $\theta_k$ nonnegative, $f_T$ is given by \eqref{e:fT} with $F_T=T^{1+\beta}$ and $g_T$ is as in Proposition \ref{prop:Laplace}. Then 
 \begin{equation}
 H(g_T(t))\le f_T(t), \qquad \textrm{ for any}\ t\ge 0.
 \label{e:Hg_le_f}
\end{equation}
Moreover, there exists $T_0\ge 1$ (which may depend on $f$)  and $C>0$ such that for any $T\ge T_0$ we have
\begin{equation}
Tg_T(tT)\le C f^{1/(1+\beta)}(t), \qquad t\ge 0.
\label{e:Tg_T_bound}
\end{equation}
 Thus
  \begin{equation}
 \limsup_{T\to \infty}\int_0^\infty g_T(t)dt<\infty.
\label{e:limsup_g_T}
 \end{equation}
 b) If the number of offspring of a single particle has finite variance and $F_T=T^2$ then \eqref{e:Hg_le_f}-\eqref{e:limsup_g_T}
  follow with $\beta=1$ and $F_T=T^2$.
\end{lem}
\proof a)
We start with formula  \eqref{e:g_eqH}. By adding and subtracting $H(g_T(t))$ to both sides  and  then regrouping, we obtain
\begin{align*}
 H(g_T(t))=&1-e^{-f_T(t)}-e^{-f_T(t)}\left(g_T(t)-H_T(g_T(t))
 -\int_0^\infty g_T(t)\varphi(s)ds+H(\int_0^\infty g_T(t)\varphi(s)ds)
 \right)\\
 &-(1-e^{-f_T(t)})\left(g_T(t)-H(g_T(t))\right).
\end{align*}
From the definition of $H$ in \eqref{e:Hdef} and the fact that $G$ is a probability generating function it follows directly that the function $s\mapsto s-H(s)=1-G(1-s)$ is nonnegative and nondecreasing on $[0,1]$. Using this and
 \eqref{e:gT_ineq},  we see that the terms in parentheses involving $g_T$ are nonnegative, therefore 
\begin{equation*}
 H(g_T(t))\le 1-e^{f_T(t)}\le  f_T(t), \qquad \textrm{ for any}\ t\ge 0.
\end{equation*}
This proves  \eqref{e:Hg_le_f}.

Recalling, \eqref{e:FT} and \eqref{e:hR}  we have that 
\begin{equation*}
 h_T(t)\le \frac{C}{T^{1+\beta}}+\frac{C}{T^{1+\beta}} I_R(Ta)
\end{equation*}
where $a$ is given by \eqref{e:a}, that is $\supp{f_T}\subset[0,Ta]$. 
By \eqref{e:ineq_g_f} and \eqref{e:IR_a}
\begin{equation*}
 \sup_{t\ge 0}g_T(t)\le \frac C{T^{\beta}}\to 0, \qquad \textrm{as }\ T\to \infty.
 \label{e:3:22a}
\end{equation*}
By Assumption \ref{ass:beta_general} (cf. \eqref{e:H_bound} and \eqref{e:limW}) there exists $\varepsilon>0$ such that for any $0\le x\le  \varepsilon$ we have $H(x)x^{-1-\beta}\ge \frac 12 K$. Let $T_0$ be sufficiently large, so that $\sup_{T\ge T_0}\sup_{t\ge 0}g_T(t)\le \varepsilon$, then
\begin{equation*}
Tg_T(tT)\le \frac 2K T \left(H(g_T(t))\right)^{1/1+\beta}\le \frac 2K T\left(f_T(tT)\right)^{1/1+\beta}=\frac 2K f^{1/(1+\beta)}(t), \qquad t\ge 0, T\ge T_0.
\end{equation*}
This proves \eqref{e:Tg_T_bound}. 
Note that this in particular implies that 
\begin{equation*}
 \limsup_{T\to \infty}\int_0^\infty g_T(t)dt=\limsup_{T\to \infty}\int_0^\infty Tg_T(tT)dt\le C\int_0^\infty f^{1/1+\beta}(t)dt<\infty.
 \end{equation*}
 Hence \eqref{e:limsup_g_T} is also proved.
 
 \medskip
 Part b) follows with minimal changes. In this case, if we denote by $\rho$ the generic random variable of the number of offspring of a single particle, then if $\rho$ has a finite variance, then it is easy to see that
 \begin{equation}
  \lim_{x\to 0}\frac{H(x)}{x^2}=\frac 12 E\rho (\rho-1)=\frac 12\Var \rho. \label{e:3.21a}
 \end{equation}
 The last equality follows from the fact that we assume that $E\rho=1$. Also, $H(x)$ is bounded by $Cx^2$ for some constant $C>0$. We also have \eqref{e:H_bound_2} with $\beta=1$.
 \qed
 
 \begin{lem}\label{lem:lem2}
 a)
Under the assumptions of Lemma \ref{lem:lem1} a) we have that
   \begin{equation}
   g_T(t)=\int_0^\infty\left(f_T(t+s)-H(g_T(t+s))\right) R(s)ds+Q_T(t)
   \label{e:Q_T}
  \end{equation}
where $Q_T$ has its support in $[0,Ta]$, where $a$ is given by \eqref{e:a} and  
\begin{equation}
 \liminf_{T\to \infty}( \inf_{t\ge 0} TQ_T(t))\ge 0
 \label{e:Q_T_1}
\end{equation}
Moreover, if additionally Assumption \ref{ass:phi_alpha} holds with $\beta=\alpha$, then
\begin{equation}
\limsup_{T\to \infty}(\sup_{t\ge 0}\abs{TQ_T(t)})=0.
 \label{e:Q_T_2}
\end{equation}

b) Under the assumptions of Lemma \ref{lem:lem1} b) and if additionally $\varphi$ satisfies \eqref{e:phi_x_integrable} then \eqref{e:Q_T} and \eqref{e:Q_T_2} are satisfied. 
\end{lem}
\proof
a)
By \eqref{e:gT_VT_2} and \eqref{e:def_VT} we can write 
\begin{equation}
 g_T(t)=f_T(t)-H_T(g_T(t))+\int_0^\infty \left(f_T(t+s)-H(g_T(t+s))\right)R(s)ds
+S_T(t)+\int_0^\infty S_T(t+s)R(s)ds,
\label{e:gT_ST}
 \end{equation}
where 
\begin{align}
 S_T(t)=&V_T(t)-f_T(t)+H(g_T(t))
= \sum_{k=1}^4S_{T,k}(t)
\label{e:S_T_def}
\end{align}
and 
\begin{align}
 S_{T,1}(t)=& 1-e^{-f_T(t)}-f_T(t)\label{e:S_T_1}\\
 S_{T,2}(t)=& -(1-e^{-f_T(t)})\int_0^\infty g_T(t+s)\varphi(s)ds \label{e:S_T_2}\\
  S_{T,3}(t)=& (1-e^{-f_T(t)})H(g_{T}(t))\label{e:S_T_3}\\
   S_{T,4}(t)=& e^{-f_T(t)}\left(H(g_T(t))-H(\int_0^\infty g_T(t+s)\varphi(s) ds)\right).\label{e:S_T_4}
\end{align}
We therefore have 
\begin{equation}
 Q_T(t)=f_T(t)-H_T(g_T(t))
+S_T(t)+\int_0^\infty S_T(t+s)R(s)ds,
\label{e:QT_def}
\end{equation}
Clearly the support of $Q_T$ is in  $[0,Ta]$, since this is the support of $g_T$ and $f_T$.
By \eqref{e:Hg_le_f}, \eqref{e:fT} and \eqref{e:FT}
\begin{equation}
 0\le f_T(t)-H_T(g_T(t))\le f_T(t)\le \frac{C}{T^{1+\beta}}
 \label{e:f_H_estimate}
\end{equation}
Using \eqref{e:Hg_le_f}, \eqref{e:Tg_T_bound}, \eqref{e:gT_ineq} and some elementary estimates we obtain 
\begin{align}
 \abs{S_{T,1}(t)}\le &\frac 12 f_T^2(t)\label{e:S1}\\
  \abs{S_{T,2}(t)}\le & f_T(t)g_T(t)\le C f_T^{1+\frac 1{1+\beta}}(t)\label{e:S2}\\
  \abs{S_{T,3}(t)}\le  & f_T(t)H(g_T(t))\le f_T^2(t)\label{e:S3}\\
  \abs{S_{T,4}(t)}\le  & C f_T(t).\label{e:S4}
\end{align}
Since $f_T$ is bounded, the above estimates imply
\begin{equation*}
 \abs{S_T(t)}\le C_1 f_T(t).
 \end{equation*}
But the above, \eqref{e:S_T_def} and \eqref{e:Hg_le_f} imply that 
\begin{equation*}
 V_T(t)\le f_T(t).
\end{equation*}
Using \eqref{e:H_bound_2} and \eqref{e:gT_VT} we can get a finer estimate of $\abs{S_{T,4}(t)}$, namely from \eqref{e:S_T_4} we obtain
\begin{align} \abs{S_{T,4}(t)}\le  & C g_T^\beta(t)V_T(t)
 \le C_1 f_T^{\frac \beta{1+\beta}+1}.
 \label{e:S4_alt}
\end{align}
Since $f_T(t)$ is bounded, the estimates \eqref{e:S1}-\eqref{e:S3} and \eqref{e:S4_alt} imply that
\begin{equation}
\abs{ S_T(t)}\le C f_T^{\frac {1+2\beta}{1+\beta}}(t)=C\frac 1{T^{1+2\beta}}f^{\frac {1+2\beta}{1+\beta}}(\frac tT).
 \label{e:S_alt}
\end{equation}
From \eqref{e:S1}-\eqref{e:S3}
We also have
\begin{equation}
 \sup_{t\ge0}\sum_{j=1}^3\abs{S_{T,j}(t)}\le \frac C{T^{2+\beta}}
 \label{e:S_1_3_a}
\end{equation}
hence there exists $T_0$ such that for $T\ge T_0$ we have
\begin{equation}
  \sup_{t\ge 0}T\int_0^\infty \sum_{j=1}^3\abs{S_{T,j}(t+s)}R(s)ds\le
 \frac C{T^{1+\beta}}I_R(aT)\le C_1\frac {a}{T^\beta}.
\label{e:S_1_3}
 \end{equation}
 where in the last estimate we have used \eqref{e:IR_a}. From the part \eqref{e:H_bound_2} of Assumption \ref{ass:beta_general}
and the fact that in \eqref{e:gT_VT} $V_T$ is nonnegative it follows that the 
 term $S_{T,4}$ is nonnegative. Thus \eqref{e:Q_T_1} follows from \eqref{e:QT_def}, \eqref{e:f_H_estimate}, \eqref{e:S_alt}, \eqref{e:S_1_3_a} and \eqref{e:S_1_3}.

 To verify \eqref{e:Q_T_2} we observe that we have a better estimate of  $I_R(aT)$ following from \eqref{e:IR}. In this case, using also \eqref{e:S_alt}, we see that there exists $T_0\ge 1$ such that for all $T\ge T_0$ we have
  \begin{equation}
  T\int_0^\infty \abs{S_T(t+s)}R(s)ds\le C\frac{1}{T^{2\beta}}I_R(aT)
  \le C_1 \frac {a^\beta}{T^\beta}, \qquad t\ge 0.
  \label{e:3.40a}
 \end{equation}
\eqref{e:Q_T_2} now follows from \eqref{e:QT_def}, \eqref{e:f_H_estimate}, \eqref{e:S_alt} and \eqref{e:3.40a}.

Part b) follows in the same way if we set $\beta=1$. In \eqref{e:3.40a} we now use \eqref{e:IR_a_1}.
 \qed
 
\subsection{Proof of Theorem \ref{thm:alpha_ge_beta}}
First we will show convergence of finite dimensional distributions and then tightness in $M_1$ topology of $D([0,\infty))$. 

\textbf{Convergence of finite dimensional distributions}
To prove convergence of finite dimensional distributions we will use the method of Laplace transforms. Recall Proposition \ref{prop:Laplace}. By the standard argument, it is enough to show that for any $f$ of the form \eqref{e:f_sum} with all $\theta_k$ nonnegative we have
\begin{equation}
\lim_{T\to \infty} E \exp\{-\<N,f_T\>\}=\exp\{-
{K^{-\frac 1{1+\beta}}}
\int_0^\infty(f(t))^{\frac 1{1+\beta}}dt\},
\label{e:Lapl_alpha_beta}
\end{equation}
where
$f_T$ is defined in \eqref{e:fT}. Recalling \eqref{e:Laplace_2} and the form of the constant $K_1$, this amounts to proving that
\begin{equation}
 \lim_{T\to\infty}\int_0^\infty g_T(t)dt=
 K^{-\frac 1{1+\beta}}\int_0^\infty(f(t))^{\frac 1{1+\beta}}dt
 \label{e:lim_g_beta}
\end{equation}
From \eqref{e:Tg_T_bound} and \eqref{e:Q_T_1} we have 
\begin{align*}
\infty> &\limsup_{T\to \infty}\int_0^\infty g_T(t)dt\\
\ge & \limsup_{T\to \infty }\int_0^\infty\int_0^\infty \left(f_T(t+s)-H(g_T(t+s))\right)R(s)ds +\liminf_{T\to \infty}\int_0^\infty TQ_T(Tt)dt\\
=&\limsup_{T\to \infty }\int_0^\infty \left(f_T(t)-H(g_T(t))\right)I_R(t)dt\\
=&\limsup_{T\to \infty }\frac 1{T^\beta}\int_0^\infty \left(f(t)- T^{1+\beta}H(g_T(tT))\right)I_R(tT)dt.
\end{align*}
Fix $\varepsilon>0$ and consider $\delta$ given by Assumption \ref{ass:beta_integrable}. Without loss of generality we may assume that $\delta+\beta\le 1$.
By \eqref{e:IR_b}
there exists $C>0$ and $T_0\ge 1$ such that for any $T$ such that $T\varepsilon\ge T_0$ and $t\ge \varepsilon$
\begin{equation*}
 I_R(Tt)\ge I_R(T\varepsilon)\ge C (T\varepsilon )^{\beta+\delta}.
\end{equation*}
Therefore, for any $\varepsilon>0$
\begin{align*}
 \infty>\limsup_{T\to \infty }C T^\delta \varepsilon^{\beta+\delta}\int_\varepsilon^\infty \left(f(t)- T^{1+\beta}H(g_T(tT))\right)dt.
\end{align*}
Since $T^\delta$ goes to infinity, this implies that for any $\varepsilon>0$ we have
\begin{equation}
 \limsup_{T\to \infty}\int_\varepsilon^\infty \left(f(t)- T^{1+\beta}H(g_T(tT))\right)dt=0.
\label{e:lim_g_T_beta}
 \end{equation}
 Also note, that since $f(t)- T^{1+\beta}H(g_T(tT)$ is nonnegative and 
 \begin{equation*}
\int_0^\varepsilon\abs{f(t)- T^{1+\beta}H(g_T(tT)} dt\le  \varepsilon                      \sup_{t\ge 0}\abs{f(t)}
\end{equation*}
\eqref{e:lim_g_T_beta} implies that
\begin{equation}
 \limsup_{T\to \infty}\int_0^\infty \left(f(t)- T^{1+\beta}H(g_T(tT))\right)dt=0.
 \label{e:limsup_f_H}
\end{equation}
We will show that \eqref{e:limsup_f_H} implies that
\begin{equation}
 \label{e:lim_int_g_T}\int_0^\infty K^{\frac 1{1+\beta}}Tg_T(Tt)dt=\int_0^\infty f^{\frac 1{1+\beta}}(t)dt.
\end{equation}
Indeed, first recall that the supports of $f$ and $g_T(T\cdot)$ are in $[0,a]$. Using 
$(x+y)^{1/(1+\beta)}\le x^{1/(1+\beta)}+ y^{1/(1+\beta)}$ for any $x,y\ge 0$, Jensen's inequality and then \eqref{e:limsup_f_H} we obtain
\begin{align}
\limsup_{T\to \infty}\int_0^\infty \abs{f^{\frac 1{1+\beta}}(t)- T\left(H(g_T(Tt))^{\frac 1{1+\beta}}\right)}dt
 \le&\limsup_{T\to \infty} \int_0^a 
 \left(f(t)-T^{1+\beta}H(g_T(Tt))\right)^{\frac 1{1+\beta}}dt\notag \\
\le &\limsup_{T\to \infty} a^{\frac \beta{1+\beta}} \left(\int_0^a 
 \abs{f(t)-T^{1+\beta}H(g_T(Tt))\right)dt}^{\frac 1{1+\beta}}\notag\\
 =&0.
 \label{e:3.33}
\end{align}
On the other hand, using the condition \eqref{e:Hg_le_f} and \eqref{e:H_bound}  (cf. Assumption \ref{ass:beta_general}) we obtain
\begin{equation}
 \int_0^\infty\abs{T(H(g_T(Tt)))^{\frac 1{1+\beta}}-K^{\frac 1{1+\beta}}(Tg_T(Tt))}dt
\le \int_0^\infty f^{\frac 1{1+\beta}}(t)K^{\frac 1{1+\beta}}\abs{\left(W(g_T(Tt))\right)^\frac 1{1+\beta}-1}dt\to 0.
\label{e:3.34}
\end{equation}
The last convergence follows by the dominated convergence theorem and \eqref{e:Tg_T_bound}.
Combining \eqref{e:3.33} and \eqref{e:3.34} finishes the proof of \eqref{e:lim_int_g_T} and of the convergence of finite dimensional distributions.

\textbf{Convergence in $D([0,\infty)$ with $M_1$ topology.} 

The proof proceeds in the standard way. Fix any sequence $T_n\to \infty$ as $n\to \infty$.  Since the limit process does not have any fixed points of discontinuity, it suffices to show convergence of $Y_{T_n}$ in law in $ D([0,M])$ equipped with $M_1$ topology for any $M>0$ (cf. \cite{Whitt}, Sec. 3.3 and 12.9). 
To this end it is enough to have convergence of finite dimensional distributions (already proved) and tightness in $M_1$ in $D([0,M])$ (see \cite{Whitt} Theorem 11.6.6).

To prove tightness in $M_1$ topology we use  Theorem 12.12.3 in \cite{Whitt}, see also Theorem 5 in section 3 of \cite{Kern}, where the same theorem is reformulated in a more friendly form.  We observe that since the trajectories of $Y_T$ are nondecreasing, there are no oscillations, also, $Y_T(0)=0$, therefore the conditions of Theorem 12.12.3 in \cite{Whitt} (see also Thm 3.5 in \cite{Kern}) simplify. It suffices to prove
 that for any sequence $T_n\to \infty$ and $\varepsilon>0$ we have
\begin{align}
 \lim_{r\to \infty}&\limsup_{n\to\infty}P(Y_{T_n}(M)\ge r)=0\label{e:3.19a}\\
 \lim_{\delta\searrow 0}&\limsup_{n\to \infty}P(Y_{T_n}(\delta)>\varepsilon)=0\label{e:3.19b}\\
 \lim_{\delta\searrow 0}&\limsup_{n\to \infty}P(Y_{T_n}(M-)-Y_{T_n}(M-\delta)>\varepsilon)=0.\label{e:3.19c}
\end{align}
We have already shown that nondecreasing processes $Y_{T_n}$ converge in the sense of finite dimensional distributions. Also, it is known that the $\frac 1{1+\beta}$-stable subordinator is stochastically continuous. These two facts imply \eqref{e:3.19a}-\eqref{e:3.19c}. \eqref{e:3.19a} follows from tightness of $(Y_{T_n}(M))_n$. To prove \eqref{e:3.19c} we may assume that $\varepsilon<1$ and we use 
 \begin{align}
 \limsup_{n\to \infty}P(Y_{T_n}(M-)-Y_{T_n}(M-\delta)>\varepsilon)&\le 
  \limsup_{n\to \infty}P(Y_{T_n}(M)-Y_{T_n}(M-\delta)>\varepsilon)\notag\\
    &\le \limsup_{n\to \infty} \frac{E(Y_{T_n}(M)-Y_{T_n}(M-\delta))\wedge 1}{\varepsilon}\notag\\
    &=\frac{E(K^{-1}(L_{1/(1+\beta)}(M)- L_{1/(1+\beta)}(M-\delta)))\wedge 1}{\varepsilon}.
    \label{e:EL}
 \end{align}
In the last equality we have used the convergence of $2$-dimensional distributions of  the processes $Y_{T_n}$. 
By stochastic continuity of the process $L_{1/(1+\beta)}$, the term in \eqref{e:EL} converges to $0$ as $\delta\to 0$.
This finishes the proof of \eqref{e:3.19c}. \eqref{e:3.19b} is shown in the same way.

\medskip

\qed

\subsection{Proof of 
Theorem \ref{thm:critical}}
As in the proof of Theorem \ref{thm:alpha_ge_beta} we will show the convergence of finite dimensional distributions using Laplace transforms and then tightness in $M_1$ topology of the the Skorokhod space. This time the proof of convergence of finite dimensional distributions is more involved since we also have to prove existence of the limit process, and its Laplace transform is more complicated.

\textbf{Scheme of the proof.}
Recall Proposition \ref{prop:Laplace} and the notation used there. Now $F_T=T^{1+\beta}$. As before, we consider functions $f$  of the form \eqref{e:f_sum} with $\theta_k$ nonnegative for all $k=1,2,\ldots, m$ and $a$ is given by \eqref{e:a}.

We will show that for any $f$ of this form the equation \eqref{e:g_infty} has a unique nonnegative bounded solution  $v_f$ with compact support and 
we have
\begin{equation}
 \lim_{T\to \infty}\int_0^\infty g_T(t)dt=\int_0^\infty v_f(t)dt.
 \label{e:conv_int}
\end{equation}
Note that $v_f$ depends  on $\theta_1,\ldots, \theta_m$. To conclude existence of the process $Y_\infty$ and the desired convergence of finite dimensional distributions it then suffices to show that 
the function 
\begin{equation*}
(\theta_1,\ldots, \theta_k)\mapsto \int_0^\infty v_f (t)dt,
\end{equation*}
is continuous at ${\bf 0}=(0,\ldots,0)$. By \eqref{e:Laplace_2} this implies that the distributions of $(Y_{T}(t_1),\ldots, Y_T(t_k))$ are tight because the Laplace transforms converge to a function that is continuous at zero.

To prove \eqref{e:conv_int} we will use the substitution:
\begin{equation*}
 \int_0^\infty g_T(t)dt=\int_0^\infty Tg_T(Tt)dt
\end{equation*}
and we will show the uniform convergence
\begin{equation}\lim_{T\to \infty}\sup_{t\in[0,a]}\abs{Tg_T(Tt)-v_f(t)}=0.
\label{e:uniform}
\end{equation} 
\eqref{e:uniform} will also imply that $Tg_T(Tt)$ converges  to $v_f$ uniformly on $[0,\infty)$, since $Tg_T(T\cdot)$ has its support in $[0,a]$. This will prove \eqref{e:conv_int}.

We split the proof into several steps.

In Step 1 we will show that 
\begin{equation}
 Tg_T(Tt)=\int_0^\infty \left(f(t+s)-K\left(Tg_T(T(t+s))\right)^{1+\beta}\right)T^{1-\beta}R(Ts)ds +Z_T(t),
 \label{e:Z_T}
\end{equation}
where $Z_T$ has its support in $[0,a]$ and converges uniformly to $0$ as $T\to \infty$. It is also known (cf. e.g. Lemma 4.1 in \cite{AT}) that under Assumption \ref{ass:phi_alpha} with $\alpha=\beta$ the measure with density $T^{1-\beta}R(Ts)$ converges weakly to measure with density $\beta c_{\beta,\varphi} s^{\beta-1}$, therefore it becomes clear why the limit should be of the form described in the theorem, but we need to justify it.

In Step 2 we discuss the behavior of the first part of the integral on the right hand side of \eqref{e:Z_T}.

In Step 3 we  consider $\varphi$ of a particular form, 
for which $R(s)=Cs^{\beta-1}$, which  simplifies the estimates.
We  show that for any sequence $T_n\to \infty$ the functions ${T_n}g_{T_n}(T_n\cdot)$ form a Cauchy sequence in the Banach space of Borel bounded functions on $[0,a]$, and therefore converges uniformly to some bounded function $v_f$.

In Step 4 we  verify that $v_f$ obtained in Step 3 satisfies the equation \eqref{e:g_infty}. We  also show that it is continuous in $t$ variable and it is  continuous at ${\bf 0}$ as a function of $(\theta_1,\ldots, \theta_m)$. Additionally we prove uniqueness of solutions of \eqref{e:g_infty}.

In Step 5 we return to the case of general $\varphi$ and show \eqref{e:uniform}. The fact that we already have the candidate $v_f$ for the limit makes it possible to prove the convergence \eqref{e:uniform}.
This will conclude the proof of finite dimensional distributions.

Finally, in Step 6 we discuss tightness in the Skorokhod space.

\noindent\textbf{Step 1.} Proof of \eqref{e:Z_T} with $Z_T$ converging to $0$ uniformly on $[0,a]$:

By
Lemma \ref{lem:lem2} a) we have
\begin{equation}
 g_T(t)=\int_0^\infty \left(f_T(t+s)-K(g_T(t+s))^{1+\beta}\right)R(s)ds +Q_T(t)+\tilde Q_T(t),
 \label{e:g_T_tilde}
\end{equation}
where 
\begin{equation*}
 \tilde Q_T(t)=\int_0^\infty \left(K g_T^{1+\beta}(t+s)-H(g_T(t+s))\right)R(s)ds,
\end{equation*}
and $Q_T$ satisfies \eqref{e:Q_T_2}.

Recall that $f$ is bounded and has the support in $[0,a]$ and that by \eqref{e:Tg_T_bound} there exists $C>0$ and $T_0\ge 1$ such that for all  $T\ge T_0$ we have 
\begin{equation}
 \sup_{u\ge 0}g_T(u)\le \frac{C}{T}.
 \label{e:g_bound_2}
\end{equation}
By \eqref{e:H_bound}, which is a part of Assumption \ref{ass:beta_general}, we obtain that for  all $T\ge T_0$  we have  
\begin{equation*} 
 T\tilde Q_T(t)\le 
 \frac C{T^\beta}\int_0^{Ta}\abs{1-W(g_T(t+s))}R(s) ds.
\end{equation*}
Using again \eqref{e:g_bound_2}, the fact that $W$ is bounded and $W(x)$ converges to $1$ as $x\to 0$ and \eqref{e:IR}
we see that $\tilde Q_T$ also has property \eqref{e:Q_T_2}. Hence \eqref{e:Z_T} follows with 
\begin{equation*}Z_T(t)=T\left(Q_T(Tt)+\tilde Q(Tt)\right)
\end{equation*}
uniformly convergent to $0$ as $T\to \infty$.

\noindent
\textbf{Step 2.}
From  Lemma 4.1 in \cite{AT}, recalling that now   $\alpha=\beta$, we already know that 
\begin{equation}
\lim_{T\to \infty} \int_0^\infty f(t+s)T^{1-\beta}R(Ts)ds=\beta c_{\beta,\varphi}\int_0^\infty f(t+s)s^{\beta-1}ds.
\label{e:conv_f}
\end{equation}
We want to also show that this convergence is uniform in $t\in[0,a].$
Since $f$ is a step function \eqref{e:f_sum}, by triangle inequality, it suffices to consider each of the terms separately. We have
\begin{align*}
 \sup_{t\in[0,a]}&\abs{\int_0^\infty T^{1-\beta}\ind_{[0,t_k]}(t+s)R(Ts)ds-c_{\beta,\varphi}\int_0^\infty \ind_{[0,t_k]}(t+s)s^{\beta-1}ds}\\
 &=\sup_{t\in[0,a]}\abs{T^{-\beta}I_R((t_k-t)T)-c_{\beta,\varphi}(t_k-t)^\beta}\ind_{[0,t_k]}(t)\\
 &\le \sup_{t\in[0,a]}\abs{T^{-\beta}I_R(Tt)-\beta c_{\beta,\varphi} t^\beta}.
\end{align*}
By \eqref{e:IR},
for any $\varepsilon>0$ there exists $T_0$ such that if $Tt\ge T_0$ then $\abs{(Tt)^{-\beta}I_R(Tt)-c_{\beta,\varphi}}\le \frac \varepsilon{a^\beta}$. By considering separately the cases $t< \frac T{T_0}$ and $t> \frac T{T_0}$ we obtain
\begin{equation*}
 \sup_{t\in[0,a]}\abs{T^{-\beta}I_R(Tt)-c_{\beta,\varphi} t^\beta}
 \le T^{-\beta}I_R(T_0)+c_{\beta,\varphi} \left(\frac{T_0}T\right)^\beta + \varepsilon
\end{equation*}
This shows that the  convergence in \eqref{e:conv_f} is uniform in $t$.

\noindent\textbf{Step 3.}
Now we temporarily assume that $\varphi$ is of a particular form, 
\begin{equation}
 \varphi(t)=\lambda t^{\beta-1}E_{\beta,\beta}(-\lambda t^\beta), \qquad t\ge 0
 \label{e:Mittag_Leffler_varphi}
\end{equation}
for some parameter $\lambda>0$, where 
$E_{\beta,\beta}$ is the Mittag-Leffler function
\begin{equation*}
 E_{\beta,\beta}(t)=\sum_{k=0}^\infty \frac{t^k}{\Gamma(\beta+\beta k)}.
\end{equation*}
With this choice of $\varphi$ we have $R(s)=\beta c_{\beta,\varphi} s^{\beta-1}$ (see Section 2.4.1 of  \cite{HorstXu}), which simplifies the argument.
By varying $\lambda$, we see that   $c_{\beta,\varphi}$ may be an arbitrary positive number.
Then we discuss existence and uniqueness of solutions \eqref{e:g_infty}.

We will show that for any sequence $T_n\to \infty$ the functions
$T_ng_{T_n}(T_n\cdot)$ form a Cauchy sequence in the Banach space of Borel bounded functions on $[0,a]$ equipped with the supremum norm, and therefore the sequence converges uniformly. Note that the functions $T_ng_{T_n}(T_n\cdot)$ vanish outside $[0,a]$.

Fix any $T_j,T_l\ge 1$,  and denote
\begin{equation*}
 A_{T_j,T_l}(t):=\abs{T_j g_{T_j}(T_jt)-T_lg_{T_l}(T_lt)}.
\end{equation*}
By
\eqref{e:Z_T}, the uniform convergence of $Z_T$ to $0$, \eqref{e:conv_f} and the fact that now $T^{1-\beta}R(Ts)=R(s)=\beta c_{\beta,\varphi}$ we have
\begin{equation*}
A_{T_j,T_l}(t)\le B(T_j,T_l)
+K\beta c_{\beta,\varphi}\int_0^\infty \abs{(T_jg_{T_j}(T_j(t+s))^{1+\beta}- (T_lg_{T_l}(T_l(t+s)))^{1+\beta}}s^{\beta-1}ds
 \end{equation*}
where
\begin{equation}
 \lim_{T_j,T_l\to \infty} B(T_j,T_l)=0.
 \label{e:B_conv}
\end{equation}
Using also a trivial estimate 
\begin{equation}\abs{x^{\beta+1}-y^{\beta+1}}\le (1+\beta)(x^\beta+y^\beta)\abs{x-y}\quad \textrm{for}\ x,y\ge 0 
 \label{e:x_y_beta}
\end{equation}
we obtain 
\begin{align}
 A_{T_j,T_l}(t)\le&
  B(T_j,T_l)+C_1\int_0^{a-t} A_{T_j,T_l}(t+s)s^{\beta-1}ds,
  \label{e:60}
\end{align}
for some positive constant $C_1$ and all $t\in [0,a]$ ($A_{T_j,T_l}$ vanishes outside $[0,a]$). This allows to use techniques similar to the proof of the Gronwall's lemma to verify that $A_{T_j, T_l}$ converges to $0$ uniformly on $[0,a]$ as $T_j,T_l\to \infty$.

Let $\sigma>0$, to be chosen later. Then
\begin{align*}
\sup_{t\in [0,a]} e^{-\sigma (a-t)} A_{T_j,T_l}(t)\le& B(T_j,T_l)+C_1\sup_{t\in[0,a]}\int_0^{a-t}{ e^{-\sigma (a-t-s)} A_{T_j,T_l}(t+s)}
e^{-\sigma s}s^{\beta-1}ds\\
\le &B(T_j,T_l)+C_1\left(\sup_{t\in[0,a]} e^{-\sigma (a-t)} A_{T_j,T_l}(t)\right)\int_0^\infty
e^{-\sigma s}s^{\beta-1}ds
\end{align*}
We can choose $\sigma$ large enough, so that $\kappa(\sigma,\beta):=C_1\Gamma(\beta)/\sigma^\beta<1$ and then we obtain
\begin{equation*}
 e^{-\sigma a}\sup_{t\in[0,a]}A_{T_j,T_l}(t)\le 
 \sup_{t\in[0,a]}e^{-\sigma(a-t)}A_{T_j,T_l}(t) \le\frac{B(T_j,T_l)}{1-\kappa(\sigma,\beta)}.
\end{equation*}
Therefore, by \eqref{e:B_conv}
\begin{equation*}
 \lim_{T_j,T_l\to \infty}A_{T_j,T_l}(t)=0.
\end{equation*}
This shows that in this special case, $Tg_T(T\cdot)$  is Cauchy in the space of Borel bounded functions on $[0,a]$, therefore it converges uniformly on $[0,a]$ to some $v_f$, and also uniformly on $[0,\infty)$ if we set $v_f(t)=0$ for $t>a$. The limit $v_f$ is bounded since $Tg_T$ were bounded uniformly in $T$.

We will now show that $v_f$ satisfies  the equation \eqref{e:g_infty},
and it is the unique bounded function with bounded support that satisfies \eqref{e:g_infty}.

Using \eqref{e:Z_T}, the fact that $Z_T$ converges uniformly to $0$  and the uniform convergence in \eqref{e:conv_f} we only have to show that
\begin{equation*}
 \sup_{t\in[0,a]}K\int_0^\infty \abs{Tg_T(T(t+s))^{1+\beta}-(v_f(t+s))^{1+\beta}}c_{\beta,\varphi} s^{\beta}=0,
\end{equation*}
but this is clear by using again the estimate \eqref{e:x_y_beta} and the fact that $Tg_T$ is bounded independently of $T$ and $v_f$ is also bounded.
This proves that $v_f$ satisfies the equation \eqref{e:g_infty}.

Uniqueness follows in a very similar way to the proof of the Cauchy condition above. It is enough to take two solutions $v_f$ and $v_f^*$ and estimate the difference in the similar way as it was done for $A_{T_j,T_l}$ above. We skip this obvious argument.

Also, the same argument shows that 
the function $(\theta_1,\ldots, \theta_m)\mapsto \int_0^\infty v_f(t)dt$ is continuous at zero.
Since $v_f$ is also bounded and has compact support, the Laplace transforms of $(Y_T(t_1),\ldots, Y_T(t_m))$ converge to a function that is continuous at $\bf 0$, thus proving convergence of finite dimensional distributions to those of the process of $Y_\infty$.

\medskip

We will now show that $v_f$ is continuous in the $t$ variable. This will be useful in the case of general $\varphi$. 
By uniqueness we already know that if $f$ has its support in $[0,a]$, then so does $v_f$.
Let $0\le u\le t\le a$. 
From the equation \eqref{e:g_infty} we obtain
\begin{align*}
 v_f(t)-v_f(u)=&\beta c_{\beta,\varphi}\int_{t-u}^\infty\left(f(u+s)-K(v_f(u+s))^{1+\beta}\right)\left((s-(t-u))^{\beta-1}-s^{\beta-1}\right)ds \\
 &-\beta c_{\beta,\varphi}\int_0^{t-u}\left(f(u+s)-K(v_f(u+s))^{1+\beta}\right)s^{\beta-1}ds.
\end{align*}
Since both $f$ and $v_f$ are bounded with the support in $[0,a]$, then 
\begin{align*}
 \abs{ v_f(t)-v_f(u)}\le &C\left[\int_{(t-u)}^{(a-u)}
 \left((s-(t-u))^{\beta-1}-s^{\beta-1}\right)ds+\int_0^{t-u}s^{\beta-1}ds\right]\\
 \le & C_1 \left[(a-t)^\beta-(a-u)^\beta +2(t-u)^\beta\right]\\
 \le & C_1 2(t-u)^\beta.
\end{align*}
The continuity follows.

\textbf{Step 4.}
Having proved that the equation \eqref{e:g_infty} has a unique solution and the solution is continuous in $t$, as well as existence of the process $Y_\infty$, we now return to the case of general $\varphi$ as in the Assumption \ref{ass:phi_alpha} with $\alpha=\beta$.

By \eqref{e:Z_T} and the uniform convergence in \eqref{e:conv_f} we have
\begin{align*}
\abs{Tg_T(Tt)-v_f(t)}\le &B(T) + \abs{\int_0^\infty (Tg_T(T(t+s)))^{1+\beta}T^{1-\beta}R(Ts)ds -\int_0^\infty v_f^{1+\beta}(t+s)\beta c_{\beta,\varphi} s^{\beta-1}ds}\\
\le & B(T)+C\int_0^\infty\abs{Tg_T(T(t+s)-v_f(t+s))}T^{1-\beta}R(Ts)ds
+D(T),
\end{align*}
where $B(T)$ converges to $0$ as $t\to \infty$ and 
\begin{equation*}
D_T=\sup_{t\ge 0}\abs{\int_{0}^\infty v_f^{1+\beta}(t+s) T^{1-\beta}R(Ts)ds -\int_0^\infty v_f^{1+\beta}(t+s) \beta c_{\beta,\varphi} s^{\beta-1}ds }
\end{equation*}
Since $v_f$ is a continuous function with compact support we can approximate it by step functions and use the uniform convergence in \eqref{e:conv_f} proved for step functions to see that $D(T)$ also converges to $0$ as $T\to \infty$.
The rest of the argument is similar to the one in Step 3. (cf. \eqref{e:60})
since we can choose $\sigma$ so that 
$\sup_{T\ge0}\int_0^a e^{-\sigma s}T^{1-\beta}R(sT)ds<1$,
Here we use that
\begin{equation*}
\lim_{T\to \infty} \int_0^ae^{-\sigma s}T^{1-\beta}R(Ts)ds
=\int_0^a e^{-\sigma s}\beta c_{\beta,\varphi} s^{\beta-1}ds,
\end{equation*}
by Lemma 4.1 in \cite{AT}. 

\noindent \textbf{Step 5.}

Tightness in $M_1$ follows in the same way as in the proof of Theorem
\ref{thm:alpha_ge_beta}. We have only have to prove that $Y_\infty$ is stochastically continuous. 
Then \eqref{e:3.19a}-\eqref{e:3.19c} follow in exactly the same way as in Theorem \ref{thm:alpha_ge_beta}, thus giving tightness.

To verify the stochastic continuity of $Y_\infty$ it suffices to show that for any $r_n\to r$ we have that $(Y_\infty(r_n),Y_\infty(r))$ converges in law to $(Y_\infty(r),Y_\infty(r))$, as $n\to \infty$. Let us fix $\theta_1,\theta_2>0$ and denote 
\begin{equation*}f_n=\theta_1\ind_{[0,r_n]}+\theta_2\ind_{[0,r]},\qquad f=(\theta_1+\theta_2)\ind_{[0,r]}.
\end{equation*}
Moreover, let $b$ be such that $b=\sup_n r_n$.
We have to show that
\begin{equation}
\lim_{n\to \infty} \int_0^\infty v_{f_n}(u)du=\int_0^\infty v_f(u)du.
 \label{e:lim_f_n}
\end{equation}
All $v_{f_n}$ and $v_f$ vanish outside $[0,b]$ and for $t\le b$  
we have that
\begin{equation*}
 \abs{\beta\int_0^\infty (f_n(t+s)-f(t+s))s^{\beta-1} ds}
 =\theta_2\abs{(r_n-t)^\beta_+-(r-t)^\beta_+}\le \theta_2 \abs{r_n-r}^\beta
\end{equation*}
From the equation \eqref{e:g_infty} it follows also that 
\begin{equation*}
 v_{f_n}(t)\le \beta c_{\beta,\varphi}\int_0^\infty f_n(t+s)s^{\beta-1}
\end{equation*}
and a similar inequality also holds for $f$, therefore $v_{f_n}$ and $v_f$ are uniformly bounded by the same constant. Hence
by equation \eqref{e:g_infty} we have 
\begin{equation*}
 \abs{v_{f_n}(t)-v_f(t)}\le C\abs{r_n-r}^\beta +C\int_0^{b-t}
 \abs{v_{f_n}(t+s)-v_n(t+s)}s^{\beta-1}, \qquad \textrm{for}\ t\le b. 
\end{equation*}
Again, the same argument as in Step 3 shows that $v_{f_n}$ converges uniformly to $v_f$, thus \eqref{e:lim_f_n} holds, since all of these functions are supported on $[0,b]$. This finishes the proof of stochastic continuity of $Y_\infty$ and of convergence in $M_1$.
\qed

\subsection{Proof of Theorem \ref{thm:weakly_critical}}
The proof is almost the same as that of Theorem \ref{thm:critical} with only minor changes, therefore we skip it. We only indicate the main changes: in Step 1 we use parts b) of Lemmas \ref{lem:lem1} and \ref{lem:lem2}; in Step 2 the fact that now 
$\lim_{T\to \infty } T^{-1}I_R(T)=\frac 1{m_{\varphi}}$; in Step 3 we can take $\varphi$ to be the density of an exponential distribution with parameter $\lambda$. We omit the details. Note also that in the present case the proof may be  simplified, since we already know that the equation \eqref{e:g_infty_2} has a unique bounded nonnegative solution which follows directly from Theorem 4.2 in \cite{Zenghu_Li}. This follows since \eqref{e:g_infty_2} is equivalent to the fact that $u_f(r)=v_f(a-r)$ satisfies (4.8) in \cite{Zenghu_Li} with $t=a$. We also already know that the process $Y^*_\infty$ with the given Laplace multidimensional Laplace transform exists (see Remark \ref{rem:weakly_critical}).\qed 

\subsection {Proof of representation \eqref{e:r1} and \eqref{e:r2} in Remark \ref{rem:weakly_critical}}
Let us consider first the case $\sigma_G=1$ (as for the not marked Hawkes process). From Theorem 2.11 in \cite{HorstXu} we have that the limit process described by \eqref{e:Lapl_lim_critical_2} and \eqref{e:g_infty_2} and $\sigma_G=1$ is the same as the one in Theorem 2.10 of \cite{HorstXu}. Recalling \eqref{e:a} above and   (2.15) in \cite{HorstXu}, in the notation of \cite{HorstXu} it is enough to set 
$T:=a$, $w([0,t]):=-\int_0^t f(a-s)ds$, $g:=0$. Then $V^*$ in (2.16) of \cite{HorstXu} corresponds to 
\begin{equation*}
 V^*(t)=-v_f^*(a-t).
\end{equation*}
To obtain the analogous result in the case of arbitrary $\sigma_G$ it  suffices to observe that $\tilde v_f=\sigma_G v_f^*$ satisfies the equation
\begin{equation*}
 \tilde v_f(t)=\frac {\sigma_G}{m_\varphi}\int_0^\infty \left(f(t+s)-\frac 12 (\tilde v_f(t+s))^2\right) ds.
\end{equation*}
Moreover, the Laplace transform in \eqref{e:Lapl_lim_critical_2} is of the form
\begin{equation*}
 \exp\left\{-\frac \mu{\sigma_G}\int_0^a\tilde v_f(t)dt\right\}.
\end{equation*}
Hence we can use the equivalence stated above for $\sigma_G=1$ and new $\tilde \mu=\frac \mu{\sigma_G}$ and $\tilde m_\varphi=\frac {m_\varphi}{\sigma_G}$.
\qed

\bibliographystyle{plain}

\end{document}